\documentclass[reqno,a4paper]{amsart}
\usepackage{amssymb}
\usepackage{amsmath}
\usepackage{amsfonts}

\newtheorem{theorem}{Theorem}[section]
\theoremstyle{plain}

\newtheorem{corollary}{Corollary}[section]

\newtheorem{definition}{Definition}[section]

\newtheorem{lemma}{Lemma}[section]

\newtheorem{remark}{Remark}

\numberwithin{equation}{section}
\input{tcilatex}

\begin{document}
\title[]{On Better Approximation Order for the Max-Product Meyer-K\"{o}nig
and Zeller Operator}
\author{SEZ\.{I}N \c{C}\.{I}T}
\address{Department of Mathematics, Faculty of Science, Gazi University,
Ankara, Turkey}
\email{sezincit@gazi.edu.tr}
\author{OG\"{U}N DO\u{G}RU}
\address{Department of Mathematics, Faculty of Science, Gazi University,
Ankara, Turkey}
\email{ogun.dogru@gazi.edu.tr}
\dedicatory{}
\date{}
\thanks{}

\begin{abstract}
In \cite{max prod MKZ}, Bede et al. defined the max-product Meyer-K\"{o}nig
and Zeller operator. They examined the approximation and shape preserving
properties of this operator, and they found the order of approximation to be 
$\frac{\sqrt{y}\left( 1-y\right) }{\sqrt{m}}$ by the modulus of continuity
and claimed that this order of approximation could only be improved in
certain subclasses of the functions. In contrast to this claim, we
demonstrate that we can obtain a better order of approximation without
reducing the function class by the classical modulus of continuity. We find
the degree of approximation to be $\frac{\left( 1-y\right) y^{\frac{1}{%
\alpha }}}{m^{1-\frac{1}{\alpha }}}$, $\alpha =2,3,...$ . Since $1-\frac{1}{%
\alpha }$ tends to $1$ for enough big $\alpha $, we improve this degree of
approximation.
\end{abstract}

\maketitle

\section{Introduction}

In \cite{approx by pseudo}, \cite{max-prod-shepard}, Bede et al. introduced
nonlinear operators, which are max-product and max-min operators based on
the open problem presented by Gal in \cite{Gal}. These operators showed us
that not all approximation operators need to be linear and also that the
algebraic structure used need not be the usual algebraic structure of the
real numbers. In this current study, we are concerned with max-product
operators that are generated by replacing probability by maximum set
functions, starting from the idea of obtaining linear positive operators by
probabilistic approximation. The construction of nonlinear operators takes
into account the possibility of transforming the algebraic structure from
the linear one to another more general one \cite{approx by pseudo}, \cite%
{max-prod kitap}. Therefore, we generalize the results from classical
approximation theory.

The study of nonlinear max-product operators, starting with the first
mention of the max-product Shepard operators, has been extensively studied
for important classes of linear and positive operators. The approximation
properties and shape preserving properties can be found in papers \cite%
{max-prod-shepard}, \cite{max-prod-berns}, \cite{max-prod-Szasz}, \cite%
{max-prod-berns-szazs}, \cite{Max-prod- Bleimann-Butz-Hahn}, \cite{max-prod
Baskakov}, \cite{max prod MKZ}, \cite{max-prod kitap}.

Before max-product Meyer K\"{o}nig and Zeller operator, which are the focus
of this paper, let us recall the classical Meyer K\"{o}nig and Zeller
operator introduced in \cite{Meyer Konig Zeller} as 
\begin{equation}
M_{m}(g;y)=\sum_{r=0}^{\infty }\binom{m+r}{r}y^{r}\left( 1-y\right)
^{m+1}g\left( \frac{r}{r+m+1}\right) ,  \label{classicMKZ}
\end{equation}%
where $y\in \left[ 0,1\right) $, $m\in 
\mathbb{N}
$. Then, in \cite{cheneysharma}, Cheney and Sharma made a slight
modification by replacing $g\left( \frac{r}{m+r+1}\right) $ by $g\left( 
\frac{r}{m+r}\right) $. Additionally, the classical Meyer-K\"{o}nig and
Zeller operator it can be written as%
\begin{equation*}
M_{m}(g;y)=\frac{\sum_{r=0}^{\infty }\binom{m+r}{r}y^{r}\left( 1-y\right)
^{m+1}g\left( \frac{r}{r+m+1}\right) }{\sum_{r=0}^{\infty }\binom{m+r}{r}%
y^{r}\left( 1-y\right) ^{m+1}},
\end{equation*}%
since $M_{m}(e_{0};y)=1.$ In \cite{max prod MKZ}, Bede at al. replaced the
sum operator $\sum $ by the maximum operator $\bigvee $, and gave the
definition of the max-product Meyer-K\"{o}nig Zeller operator as follows:%
\begin{equation}
Z_{m}^{(M)}(g)(y):=\frac{\bigvee\limits_{r=0}^{\infty }z_{m,r}\left(
y\right) \,g\left( \frac{r}{m+r}\right) }{\bigvee\limits_{r=0}^{\infty
}z_{m,r}\left( y\right) },\text{ }y\in \lbrack 0,1),\text{ }m\in 
\mathbb{N}
\text{ }  \label{maxproductMKZ}
\end{equation}%
where $z_{m,r}\left( y\right) =\binom{m+r}{r}y^{r}.$ Also see the degree of
approximation and some shape-preserving properties of $Z_{m}^{(M)}(g)(y)$\
operators are analysed in \cite{max prod MKZ}.

Our aim in this paper is to obtain a better approximation degree than \cite%
{max prod MKZ} with the modulus of continuity without reducing the classes
of functions. In particular, the degree of approximation for the $%
Z_{m}^{(M)}(g)(y)$ operator can be found (via the modulus of continuity) as $%
\omega \left( g;\sqrt{y}\left( 1-y\right) /\sqrt{m}\right) $. In addition,
Bede et al. claimed that the order of approximation could not be improved
except for some subclasses of functions (see, for details, \cite{max prod
MKZ}). Contrary to this claime, will demonstrate that a better degree of
approximation can be gained using the classical modulus of continuity. Our
motivation in obtaining this improved approximation is that we do not have
to use the Cauchy-Schwarz inequality when finding the approximation of the
nonlinear operators with the modulus of continuity.

Moreover, with similar way, order of approximation can be improved for
nonlinear max-product operators whose approximation and rate of convergence
properties are examined (see, for instance,\cite{better berns}, \cite{better
BHH}, \cite{better Szasz}).

\section{Preliminaries}

In order to determine the significant conclusions, it is essential to
provide an overview of the nonlinear operators known as max-product type. In
this part, we'll review the fundamental definitions and theorems of
nonlinear operators as presented in \cite{max-prod-berns-szazs}, \cite%
{approx by pseudo}.

In max-product approximation operators, the algebraic structure is changed
from the field of real numbers to an ordered, max-product semiring of
positive reals. It has the semiring structure $\left( \mathbb{R}_{+},\vee
,\cdot \right) $ with the operations $\vee $\ (maximum) and $\cdot $
(product) on the set $\mathbb{R}_{+}$ and is called max-product algebra.

Let $J\subset \mathbb{R}$ be bounded or unbounded intervals and 
\begin{equation*}
C_{B}^{+}\left( J\right) =\left\{ g:J\rightarrow \mathbb{R}_{+}:g\text{
continuous, bounded on }J\text{ }\right\} .
\end{equation*}%
Let $L_{m}$ be the max-product operator and we consider the general format
of $L_{m}:C_{B}^{+}\left( J\right) \rightarrow C_{B}^{+}\left( J\right) ,$
as follows 
\begin{equation*}
L_{m}\left( g\right) \left( y\right) =\bigvee\limits_{j=0}^{m}K_{m}\left(
y,y_{j}\right) g(y_{j})
\end{equation*}%
where $m\in 
\mathbb{N}
$, $g\in CB_{+}(J)$, $K_{m}\left( .,y_{j}\right) \in CB_{+}\left( J\right) $
and $y_{j}\in J$, for all $j.$

These operators are positive, and fulfil the property of pseudo-linearity,%
\begin{equation*}
L_{m}\left( \alpha g\vee \beta h\right) \left( y\right) =\alpha
\,L_{m}\left( g\right) \left( y\right) \vee \beta \,L_{m}\left( h\right)
\left( y\right) ,\forall \alpha ,\beta \in 
\mathbb{R}
_{+},g,h\in C_{B}^{+}\left( J\right) .
\end{equation*}%
which is weaker than linearity. Therefore they are not linear.

The modulus of continuity is often used in classical approximation theory to
evaluate error estimates. In \cite{approx by pseudo}, the definition of the
modulus of continuity adapted for nonlinear operators is as follows (also
its main properties are given in \cite{approx by pseudo}).

\begin{definition}
\cite{approx by pseudo} Let's consider the metric space $\left( Y,d\right) ,$
which is compact, and the metric space $\left( \left[ 0,\infty \right)
,\left\vert .\right\vert \right) $, which is the usual metric of positive
real numbers. $g:Y\rightarrow \left[ 0,\infty \right) $ being a bounded
function, the modulus of continuity is present following 
\begin{equation*}
\omega \left( g,\delta \right) =\vee \left\{ \left\vert g\left( y\right)
-g\left( q\right) \right\vert ,\text{ }y,q\in y,\text{ }d\left( y,q\right)
\leq \delta \right\} .
\end{equation*}
\end{definition}

Also, let us remember the recognized definition of the classical modulus of
continuity $g\in C_{B}^{+}\left( J\right) $ and $\delta >0,$%
\begin{equation}
\omega \left( g,\delta \right) =\max \left\{ \left\vert g\left( y\right)
-g\left( q\right) \right\vert ;\text{ }y,q\in J,\text{ }\left\vert
y-q\right\vert \leq \delta \right\} .  \label{1.1}
\end{equation}

\begin{lemma}
\label{Lemma 2.1} \cite{max-prod-berns-szazs} Let's consider that the
interval $J\subset \mathbb{R}$ is bounded (unbounded) and $%
L_{m}:C_{B}^{+}\left( J\right) \rightarrow C_{B}^{+}\left( J\right) ,$ $m\in 
\mathbb{N}
$ is a sequence of operators fulfilling the following features for $g,h\in
C_{B}^{+}\left( J\right) $ and $m\in \mathbb{%
\mathbb{N}
}$, which are monotonicity and sublinearity, respectively:

$(i)$ If $g\leq h$ then $L_{m}\left( g\right) \leq L_{m}\left( h\right) $.

$(ii)$ $L_{m}\left( g+h\right) \leq L_{m}\left( g\right) +L_{m}\left(
h\right) .$

Then for every $y\in J$ we have%
\begin{equation*}
\left\vert L_{m}\left( g\right) \left( y\right) -L_{m}\left( h\right) \left(
y\right) \right\vert \leq L_{m}\left( \left\vert g-h\right\vert \right)
\left( y\right) .
\end{equation*}
\end{lemma}

\begin{corollary}
\label{Corollary2.2}\ \cite{max-prod-berns-szazs} Let's consider $%
L_{m}:C_{B}^{+}\left( J\right) \rightarrow C_{B}^{+}\left( J\right) ,$ $m\in 
\mathbb{%
\mathbb{N}
}$ be positive homogeneous in addition to conditions $(i),(ii)$ in Lemma \ref%
{Lemma 2.1} above. Also let $\left( L_{m}\right) _{m}$ fulfil the condition $%
L_{m}\left( e_{0}\right) =e_{0},$ for each $m\in \mathbb{%
\mathbb{N}
}$. Then for each $g\in C_{B}^{+}\left( J\right) $, $m\in \mathbb{%
\mathbb{N}
}$, $y\in J$ we get 
\begin{equation*}
\left\vert g\left( y\right) -L_{m}\left( g\right) \left( y\right)
\right\vert \leq \left[ 1+\frac{1}{\delta }L_{m}\left( \varphi _{y}\right)
\left( y\right) \right] \,\omega \left( g,\delta \right) .
\end{equation*}%
where $\delta >0$, $e_{0}\left( t\right) =1$ for all $t\in J,$ $\varphi
_{y}\left( t\right) =\left\vert t-y\right\vert $, for all $y\in J.$
\end{corollary}

\section{Auxiliary Results}

In this section, for the approximation theorem, we present the statements
and lemmas given in \cite{max prod MKZ}\ without proof, except for Lemma \ref%
{Lemma 4.3}. Let's note that thanks to Lemma \ref{Lemma 4.3}, which we have
revised and proved, we have been able to improve the degree of approximation
of the $Z_{m}^{(M)}(g)(y)$ operator.

\begin{remark}
\cite{max prod MKZ} $1)$ In all expressions and in the approximation theorem
we will suupose that $g:\left[ 0,1\right] \rightarrow 
\mathbb{R}
_{+}$ is continuous on $\left[ 0,1\right] $. Also, we will take $0<y<1.$
Because of $Z_{m}^{(M)}(g)(0)-g\left( 0\right) =Z_{m}^{(M)}(g)(1)-g\left(
1\right) =0$ for all $m.$

$2)$ The operator $Z_{m}^{(M)}(g)(y)$ provides all the conditions of Lemma %
\ref{Lemma 2.1} and Corollary \ref{Corollary2.2}.
\end{remark}

For each $r,s\in \left\{ 0,1,2,...\right\} $ and $y\in \left[ \frac{s}{m+s},%
\frac{s+1}{m+s+1}\right] ,$ let us determine the following expressions,

\begin{equation*}
m_{r,m,s}\left( y\right) :=\dfrac{z_{m,r}\left( y\right) }{z_{m,s}\left(
y\right) },\text{ }M_{r,m,s}\left( y\right) :=\dfrac{z_{m,r}\left( y\right) 
}{z_{m,s}\left( y\right) }\left\vert \frac{r}{m+r}-y\right\vert
\end{equation*}%
similar to \cite{max prod MKZ}. We obtain the following states:

when $r\geq s+1$ 
\begin{equation*}
M_{r,m,s}\left( y\right) =m_{r,m,s}\left( y\right) \left( \frac{r}{m+r}%
-y\right)
\end{equation*}%
when $r\leq s$ 
\begin{equation*}
M_{r,m,s}\left( y\right) =m_{r,m,s}\left( y\right) \left( y-\frac{r}{m+r}%
\right) ,
\end{equation*}%
where $z_{m,r}\left( y\right) =\binom{m+r}{r}y^{r}.$

\begin{lemma}
\label{Lemma 4.1} \cite{max prod MKZ} For all $r,$ $s\in \left\{
0,1,2,...\right\} $ and $y\in \left[ \frac{s}{m+s},\frac{s+1}{m+s+1}\right]
\ $we obtain%
\begin{equation*}
m_{r,m,s}(y)\leq 1.
\end{equation*}
\end{lemma}

\begin{lemma}
\label{Lemma 4.2} \cite{max prod MKZ} We obtain 
\begin{equation*}
\bigvee\limits_{r=0}^{\infty }z_{m,r}\left( y\right) =z_{m,s}(y),\text{ for
each }y\in \left[ \frac{s}{m+s},\frac{s+1}{m+s+1}\right] ,\text{ }s=0,1,2,...
\end{equation*}%
where $z_{m,r}\left( y\right) =\binom{m+r}{r}y^{r}.$
\end{lemma}

\begin{lemma}
\label{Lemma 4.3} Let $y\in \left[ \frac{s}{m+s},\frac{s+1}{m+s+1}\right] $
and $\alpha \in \left\{ 2,3,...\right\} .$

$(i)$ If $\ r\in \left\{ s+1,s+2,...\right\} $ is such that 
\begin{equation*}
s\leq r-\left[ r+1+\frac{\left( s+1\right) ^{2}}{m}\right] ^{\frac{1}{\alpha 
}},
\end{equation*}%
then we have%
\begin{equation*}
M_{r,m,s}(y)\geq M_{r+1,m,s}(y).
\end{equation*}%
$(ii)$ If $r\in \left\{ 0,1,...,s\right\} $ is such that 
\begin{equation*}
s\geq r+\left[ r+\frac{s^{2}}{m}\right] ^{\frac{1}{\alpha }},
\end{equation*}%
then we have 
\begin{equation*}
M_{r,m,s}(y)\geq \text{$M$}_{r-1,m,s}(y).
\end{equation*}
\end{lemma}

\begin{proof}
$(i)$ Directly we can write the following inequation from $(i)$ of Lemma $%
3.2 $ in \cite{max prod MKZ}, 
\begin{equation*}
\dfrac{M_{r,m,s}(y)}{M_{r+1,m,s}(y)}\geq \frac{\left( r+1\right) \left(
r-s-1\right) \left( m+s+1\right) }{\left( s+1\right) \left( r-s\right)
\left( r+m\right) }.
\end{equation*}%
Now, let us show that the below inequality is satisfied by using a different
and more generalized proof technique than \cite{max prod MKZ}, using the
induction method, 
\begin{equation}
\frac{\left( r+1\right) \left( r-s-1\right) \left( m+s+1\right) }{\left(
s+1\right) \left( r-s\right) \left( m+r\right) }\geq 1  \label{3.1}
\end{equation}%
holds when $s\leq r-\left[ r+1+\frac{\left( s+1\right) ^{2}}{m}\right] ^{%
\frac{1}{\alpha }}$. This condition also provides: $\left( r+1\right) +\frac{%
\left( s+1\right) ^{2}}{m}\leq (r-s)^{\alpha }.$

Firstly, when $\alpha =2$, this condition is $s\leq r-\sqrt{r+1+\frac{\left(
s+1\right) ^{2}}{m}}$ and becomes the statement demonstrated in \cite{max
prod MKZ}, i.e. inequality (\ref{3.1}) is true for $\alpha =2.$

Let us suppose that inequality (\ref{3.1}) is true for $\alpha -1.$ This
means that inequality (\ref{3.1}) is satisfied when $\left( r+1\right) +%
\frac{\left( s+1\right) ^{2}}{m}\leq (r-s)^{\alpha -1}.$

The other condition of the hypothesis $(i)$; $r\geq s+1$ i.e $r-s\geq 1,$
and $\alpha =2,3,...$ we can easily see that $\left( r-s\right) ^{\alpha
-1}\leq (r-s)^{\alpha }.$ From here we obtain that inequality (\ref{3.1}) is
satisfied for any value of $\alpha =2,3,...$. We'll finally find what we
need.

$(ii)$ Let's follow similar steps to case $i)$ as we have shown above. Once
again, let us easily write the next inequality by the case $(ii)$ of Lemma $%
3.2$ in \cite{max prod MKZ}, 
\begin{equation*}
\dfrac{\text{$M$}_{r,m,s}(y)}{\text{$M$}_{r-1,m,s}(y)}\geq \frac{s\left(
s-r\right) \left( r+m-1\right) }{r\left( m+s\right) \left( s-r+1\right) }.
\end{equation*}%
Again, by use the induction method, let it be shown that the following
inequality holds for $s\geq r+\left[ r+\frac{s^{2}}{m}\right] ^{\frac{1}{%
\alpha }}$ or $\left( s-r\right) ^{\alpha }\geq r+\frac{s^{2}}{m}$ 
\begin{equation}
\frac{s\left( s-r\right) \left( r+m-1\right) }{r\left( m+s\right) \left(
s-r+1\right) }\geq 1  \label{3.2}
\end{equation}

When $\alpha =2$, this condition is $s\geq r+\sqrt{r+\frac{s^{2}}{m}}$ and
becomes the statement demonstrated in $(ii)$ of Lemma $3.2$ in \cite{max
prod MKZ}, i.e. inequality (\ref{3.2}) is true for $\alpha =2.$

Let us now suppose that inequality (\ref{3.2})holds true for $\alpha -1$.
This means that inequality (\ref{3.2}) is satisfied when $\left( s-r\right)
^{\alpha -1}\geq r+\frac{s^{2}}{m}.$

Here, we need assume that $r\leq s-1$, because if $r=s$ is, the other
condition in the hypothesis, is not provided.

Since $r\leq s-1$ and $\alpha =2,3,...$ it follows $\left( s-r\right)
^{\alpha -1}\leq \left( s-r\right) ^{\alpha }.$ is true for every $\alpha .$
Therefore, for each value $\alpha =2,3,...$ we have shown that the
inequality (\ref{3.2}) is provided when $\left( s-r\right) ^{\alpha }\geq r+%
\frac{s^{2}}{m}.$
\end{proof}

\section{Approximation Results}

The main goal is to improve the estimation of the operators $%
Z_{m}^{(M)}(g)(y)$ for the function $g$ by using the modulus of continuity.
Due to the theorem provided, it can be observed that the degree of
approximation may be enhanced when $\alpha $ reaches a sufficiently big
value. Furthermore, using $\alpha =2$ gives approximation results that are
consistent with the findings in \cite{max prod MKZ}.

\begin{theorem}
\label{Theorem5.1} Let's consider the continuous function $g:\left[ 0,1%
\right] \rightarrow 
\mathbb{R}
_{+}$. For the operators in (\ref{maxproductMKZ}) we have%
\begin{equation*}
\left\vert Z_{m}^{(M)}(g)(y)-g(y)\right\vert \leq (1+8\left( 1-y\right) 
\text{ }y^{\frac{1}{\alpha }})\text{ }\omega \left( g;\frac{1}{m^{1-\frac{1}{%
\alpha }}}\right) ,\text{ }y\in \left[ 0,1\right] ,\text{ }m\geq 4,
\end{equation*}%
where $\alpha =2,3,...$ and $\omega \left( g;\delta \right) $ is the
classical modulus of continuity defined in (\ref{1.1})
\end{theorem}

\begin{proof}
For every $y\in \left[ \frac{s}{m+s},\frac{s+1}{m+s+1}\right] $ the operator 
$Z_{m}^{(M)}\left( g\right) \left( y\right) $ fulfils the situations in
Corollary \ref{Corollary2.2}. Thus, we get the next inequality with $\varphi
_{y}\left( t\right) =\left\vert t-y\right\vert $ 
\begin{equation}
\left\vert Z_{m}^{(M)}\left( g\right) \left( y\right) -g\left( y\right)
\right\vert \leq \left[ 1+\frac{1}{\delta _{m}}Z_{m}^{(M)}\left( \varphi
_{y}\right) \left( y\right) \right] \,\omega \left( g,\delta \right) ,
\label{4.1}
\end{equation}%
Let's define the expression%
\begin{eqnarray*}
E_{m}\left( y\right) &:&=Z_{m}^{(M)}\left( \varphi _{y}\right) \left(
y\right) \\
&=&\frac{\bigvee\limits_{r=0}^{\infty }\binom{m+r}{r}y^{r}\left\vert \frac{r%
}{m+r}-y\right\vert }{\bigvee\limits_{r=0}^{\infty }\binom{m+r}{r}y^{r}},%
\text{ }y\in \left[ 0,1\right] .
\end{eqnarray*}%
Let $y\in \left[ \frac{s}{m+s},\frac{s+1}{m+s+1}\right] ,$ where $s\in
\left\{ 0,1,...\right\} $ is constant and arbitary. Using Lemma \ref{Lemma
4.2} we may write that 
\begin{equation*}
E_{m}\left( y\right) =\max\limits_{r=0,1,...}\left\{ M_{r,m,s}\left(
y\right) \right\} .
\end{equation*}%
Initially, let we examine the $s=0$ case, $M_{r,m,0}\left( y\right) =\binom{%
m+r}{r}y^{r}\left\vert \frac{r}{m+r}-y\right\vert $ where $r=0,1,2,...$ and $%
y\in \left[ 0,\frac{1}{m+1}\right] $ and $\alpha =2,3,...$.

When $r=0$, we get%
\begin{eqnarray*}
M_{0,m,0}\left( y\right) &=&y=y^{\frac{1}{\alpha }}y^{1-\frac{1}{\alpha }%
}\leq \frac{y^{\frac{1}{\alpha }}}{\left( m+1\right) ^{1-\frac{1}{\alpha }}}
\\
&\leq &\frac{y^{\frac{1}{\alpha }}}{m^{1-\frac{1}{\alpha }}}=\frac{\left(
1-y\right) \text{ }y^{\frac{1}{\alpha }}}{m^{1-\frac{1}{\alpha }}}\frac{1}{%
\left( 1-y\right) } \\
&\leq &\frac{\left( 1-y\right) \text{ }y^{\frac{1}{\alpha }}}{m^{1-\frac{1}{%
\alpha }}}\frac{m+1}{m}\leq 2\frac{\left( 1-y\right) \text{ }y^{\frac{1}{%
\alpha }}}{m^{1-\frac{1}{\alpha }}}.
\end{eqnarray*}%
In the above inequality we used the following information. $0\leq y\leq 
\frac{1}{m+1}$ or $1-\frac{1}{m+1}\leq 1-y\leq 1,$ then we have $\frac{1}{1-y%
}\leq \frac{m+1}{m}\leq 2.$

And when $r=1,$ we get 
\begin{equation*}
M_{1,m,0}\left( y\right) =\binom{m+1}{1}y\left\vert \frac{1}{m+1}%
-y\right\vert \leq y\leq 2\frac{\left( 1-y\right) \text{ }y^{\frac{1}{\alpha 
}}}{m^{1-\frac{1}{\alpha }}},
\end{equation*}%
when $r=2$, we obtain%
\begin{eqnarray*}
M_{2,m,0}\left( y\right) &=&\binom{m+2}{2}y^{2}\left\vert \frac{2}{m+2}%
-y\right\vert \leq \frac{\left( m+1\right) \left( m+2\right) }{2}y^{2}\frac{2%
}{m+2} \\
&=&\left( m+1\right) \text{ }y^{2}\leq \left( m+1\right) \text{ }y\frac{1}{%
m+1}=y\leq 2\frac{\left( 1-y\right) \text{ }y^{\frac{1}{\alpha }}}{m^{1-%
\frac{1}{\alpha }}}.
\end{eqnarray*}%
For $r\geq 2,$ Lemma \ref{Lemma 4.3} $(i)$, it clearly following that we get 
$M_{r,m,0}\left( y\right) \geq M_{r+1,m,0}\left( y\right) $, which implies
that $E_{m}\left( y\right) =\max\limits_{r\in \left\{ 0,1,2\right\} }\left\{
M_{r,m,0}\left( y\right) \right\} ,$ $y\in \left[ 0,\frac{1}{m+1}\right] .$
Thus, we get~an upper estimate for $E_{m}\left( y\right) \leq 2\frac{\left(
1-y\right) \text{ }y^{\frac{1}{\alpha }}}{m^{1-\frac{1}{\alpha }}}$, for
each $y\in \left[ 0,\frac{1}{m+1}\right] $, when $s=0.$

At the moment it is time to look for an upper estimate for each $%
M_{r,m,s}\left( y\right) $ for $s=1,2,...$ , $y\in \left[ \frac{s}{m+s},%
\frac{s+1}{m+s+1}\right] $ and $r\in \left\{ 0,1,...\right\} $.

In fact, if we may show the next one inequality 
\begin{equation}
M_{r,m,s}\left( y\right) \leq 8\frac{\left( 1-y\right) \text{ }y^{\frac{1}{%
\alpha }}}{m^{1-\frac{1}{\alpha }}},  \label{4.2}
\end{equation}%
where $\alpha =2,3,...$ and for each $y\in \left[ \frac{s}{m+s},\frac{s+1}{%
m+s+1}\right] ,$ $r=0,1,2,...$ this is exactly equivalent to 
\begin{equation*}
E_{m}\left( y\right) \leq 8\frac{\left( 1-y\right) \text{ }y^{\frac{1}{%
\alpha }}}{m^{1-\frac{1}{\alpha }}},\text{ for all }y\in \left[ 0,1\right] 
\text{ and }m\geq 4.
\end{equation*}%
This allows us to complete the proof by taking $\delta _{m}=\frac{1}{m^{1-%
\frac{1}{\alpha }}}$ in (\ref{4.1}).

In order to show the inequality (\ref{4.2}), we have just written above, let
us examine that following situations:

$1)$ $r\geq s+1;$ $2)$ $r\leq s.$

Case $1).$ Subcase $a).$ Let $r\geq s+1$ and assume first that $s\geq r-%
\left[ \left( r+1\right) +\frac{\left( s+1\right) ^{2}}{m}\right] ^{\frac{1}{%
\alpha }}$ or $r\leq s+\left[ \left( r+1\right) +\frac{\left( s+1\right) ^{2}%
}{m}\right] ^{\frac{1}{\alpha }}.$ We get 
\begin{eqnarray*}
M_{r,m,s}\left( y\right) &=&m_{r,m,s}\left( y\right) \left( \frac{r}{m+r}%
-y\right) \leq \frac{r}{m+r}-y\leq \frac{r}{m+r}-\frac{s}{m+s} \\
&\leq &\frac{s+\left[ \left( r+1\right) +\frac{\left( s+1\right) ^{2}}{m}%
\right] ^{\frac{1}{\alpha }}}{m+s+\left[ \left( r+1\right) +\frac{\left(
s+1\right) ^{2}}{m}\right] ^{\frac{1}{\alpha }}}-\frac{s}{m+s} \\
&\leq &\frac{%
\begin{array}{c}
s\left( m+s\right) +m\left[ \left( r+1\right) +\frac{\left( s+1\right) ^{2}}{%
m}\right] ^{\frac{1}{\alpha }}+s\left[ \left( r+1\right) +\frac{\left(
s+1\right) ^{2}}{m}\right] ^{\frac{1}{\alpha }} \\ 
-s\left( m+s\right) -s\left[ \left( r+1\right) +\frac{\left( s+1\right) ^{2}%
}{m}\right] ^{\frac{1}{\alpha }}%
\end{array}%
}{\left( m+s\right) \left[ m+s+\left[ \left( r+1\right) +\frac{\left(
s+1\right) ^{2}}{m}\right] ^{\frac{1}{\alpha }}\right] } \\
&=&\frac{m\left[ \left( r+1\right) +\frac{\left( s+1\right) ^{2}}{m}\right]
^{\frac{1}{\alpha }}}{\left( m+s\right) \left[ m+s+\left[ \left( r+1\right) +%
\frac{\left( s+1\right) ^{2}}{m}\right] ^{\frac{1}{\alpha }}\right] }.
\end{eqnarray*}%
Now we observe that $r+1\leq 2\left( s+1\right) .$

Let's define the following function $\phi $%
\begin{equation}
\phi \left( r\right) :=r-\left[ \left( r+1\right) +\frac{\left( s+1\right)
^{2}}{m}\right] ^{\frac{1}{\alpha }}.  \label{fi_fonk}
\end{equation}%
In fact, if we assume $r+1>2s+2$ then $r\geq 2s+2$. Since $\alpha =2,3,...$
then $\frac{1}{\alpha }-1<0$, and also $\left( r+1\right) +\frac{\left(
s+1\right) ^{2}}{m}>0.$ So, we get 
\begin{eqnarray*}
\phi ^{\prime }\left( r\right) &=&1-\frac{1}{\alpha }\left[ \left(
r+1\right) +\frac{\left( s+1\right) ^{2}}{m}\right] ^{\frac{1}{\alpha }-1} \\
&=&1-\frac{1}{\alpha \left[ \left( r+1\right) +\frac{\left( s+1\right) ^{2}}{%
m}\right] ^{1-\frac{1}{\alpha }}}>0.
\end{eqnarray*}%
Therefore the function $\phi $ is nondecreasing on $\left[ 0,\infty \right)
, $ it follows that $\phi \left( r\right) \geq \phi \left( 2s+2\right) .$
So, we obtain $s\geq r-\left[ \left( r+1\right) +\frac{\left( s+1\right) ^{2}%
}{m}\right] ^{\frac{1}{\alpha }}\geq 2s+2-\left[ 2s+3+\frac{\left(
s+1\right) ^{2}}{m}\right] ^{\frac{1}{\alpha }}$ i.e $\left[ 2s+3+\frac{%
\left( s+1\right) ^{2}}{m}\right] ^{\frac{1}{\alpha }}\geq s+2$ or $2s+3+%
\frac{\left( s+1\right) ^{2}}{m}\geq \left( s+2\right) ^{\alpha }.$ As a
result of this inequality, there is a contradiction for each $\alpha \geq 2$%
. Such that for $\alpha =2$, $2s+3+\frac{\left( s+1\right) ^{2}}{m}\geq
s^{2}+4s+4=2s+3+\left( s+1\right) ^{2}$ means $\frac{\left( s+1\right) ^{2}}{%
m}\geq \left( s+1\right) ^{2}.$ Also since $\alpha \geq 2,3,...$ this result
contradicts for every of $\alpha \geq 2.$

As a consequences%
\begin{equation*}
\left[ \left( r+1\right) +\frac{\left( s+1\right) ^{2}}{m}\right] ^{\frac{1}{%
\alpha }}\leq \left[ 2\left( s+1\right) +\frac{\left( s+1\right) ^{2}}{m}%
\right] ^{\frac{1}{\alpha }}
\end{equation*}%
also we get%
\begin{equation*}
M_{r,m,s}\left( y\right) \leq \frac{m\left[ \left( s+1\right) \left[ 2+\frac{%
\left( s+1\right) }{m}\right] \right] ^{\frac{1}{\alpha }}}{\left(
m+s\right) \left[ m+s+\left[ \left( s+1\right) \left[ 2+\frac{\left(
s+1\right) }{m}\right] \right] ^{\frac{1}{\alpha }}\right] }.
\end{equation*}%
Since $y\in \left[ \frac{s}{m+s},\frac{s+1}{m+s+1}\right] ,$ then $y\leq 
\frac{s+1}{m+s+1}$ after simple calculations it follows $m+y-1\leq \left(
1-y\right) \left( m+s\right) $ i.e $\frac{m+y-1}{1-y}\leq m+s.$ Also, since$%
\frac{s}{m+s}\leq y,$ then we have $s\left( 1-y\right) \leq my$ i.e $s\leq 
\frac{my}{1-y}.$ So we obtain $s+1\leq 2s\leq \frac{2my}{1-y}.$

Now, let's define the following function $\vartheta $%
\begin{equation*}
\vartheta \left( \xi ,s\right) :=\frac{m\xi }{\left( m+s\right) \left(
m+s+\xi \right) }.
\end{equation*}%
The function is increasing with respect to $\xi ,$ and decreasing with
respect to $s$, when $\xi =\left[ \left( s+1\right) \left[ 2+\frac{\left(
s+1\right) }{m}\right] \right] ^{\frac{1}{\alpha }}.$ So, we get%
\begin{eqnarray*}
M_{r,m,s}\left( y\right) &\leq &\frac{m\left[ 4\frac{my}{1-y}\left[ 1+\frac{%
my}{m\left( 1-y\right) }\right] \right] ^{\frac{1}{\alpha }}}{\left( \frac{%
m+y-1}{1-y}\right) \left[ \frac{m+y-1}{1-y}+\left[ 4\frac{my}{1-y}\left[ 1+%
\frac{my}{m\left( 1-y\right) }\right] \right] ^{\frac{1}{\alpha }}\right] }
\\
&=&\frac{m\text{ }4^{\frac{1}{\alpha }}\frac{\left( my\right) ^{\frac{1}{%
\alpha }}}{\left( 1-y\right) ^{\frac{2}{\alpha }}}}{\left( \frac{m+y-1}{1-y}%
\right) ^{2}+\left( m+y-1\right) 4^{\frac{1}{\alpha }}\frac{\left( my\right)
^{\frac{1}{\alpha }}}{\left( 1-y\right) ^{\frac{2}{\alpha }+1}}} \\
&=&\frac{m\text{ }4^{\frac{1}{\alpha }}\frac{\left( my\right) ^{\frac{1}{%
\alpha }}}{\left( 1-y\right) ^{\frac{2}{\alpha }}}}{\frac{\left(
m+y-1\right) ^{2}\left( 1-y\right) ^{\frac{2}{\alpha }-1}+\left(
m+y-1\right) 4^{\frac{1}{\alpha }}\left( my\right) ^{\frac{1}{\alpha }}}{%
\left( 1-y\right) ^{\frac{2}{\alpha }+1}}} \\
&=&\frac{4^{\frac{1}{\alpha }}m^{1+\frac{1}{\alpha }}\text{ }\left(
1-y\right) y^{\frac{1}{\alpha }}}{\left( m+y-1\right) ^{2}\left( 1-y\right)
^{\frac{2}{\alpha }-1}+\left( m+y-1\right) 4^{\frac{1}{\alpha }}\left(
my\right) ^{\frac{1}{\alpha }}}.
\end{eqnarray*}%
Here let's denote the function $\psi $ as 
\begin{equation*}
\psi \left( y\right) :=\left( m+y-1\right) \left[ \left( m+y-1\right) \left(
1-y\right) ^{\frac{2}{\alpha }-1}+4^{\frac{1}{\alpha }}\left( my\right) ^{%
\frac{1}{\alpha }}\right]
\end{equation*}%
then we have $\psi $ is nondecreasing when $0<y<1$. Really since, 
\begin{eqnarray*}
\psi ^{\prime }\left( y\right) &=&\left[ \left( m+y-1\right) \left(
1-y\right) ^{\frac{2}{\alpha }-1}+4^{\frac{1}{\alpha }}\left( my\right) ^{%
\frac{1}{\alpha }}\right] \\
&&+\left( m+y-1\right) \left[ 
\begin{array}{c}
\left( 1-y\right) ^{\frac{2}{\alpha }-1}+\left( m+y-1\right) \left( \frac{2}{%
\alpha }-1\right) \left( -1\right) \left( 1-y\right) ^{\frac{2}{\alpha }-2}
\\ 
+4^{\frac{1}{\alpha }}\frac{m}{\alpha }\left( my\right) ^{\frac{1}{\alpha }%
-1}%
\end{array}%
\right]
\end{eqnarray*}%
because of $\alpha \geq 2$ or $\frac{2}{\alpha }-1\leq 0$ then $\psi
^{\prime }\left( y\right) >0$ when $0<y<1.$ Thus, we have $\psi \left(
y\right) \geq \psi \left( 0\right) =\left( m-1\right) ^{2}.$ Because of $%
m\geq 4,$ then $\left( m-1\right) ^{2}\geq $ $\frac{m^{2}}{2},$ so $\psi
\left( y\right) \geq $ $\frac{m^{2}}{2}.$

Finally, for this case, using our last result, we obtain 
\begin{equation*}
M_{r,m,s}\left( y\right) \leq 2\frac{4^{\frac{1}{\alpha }}m^{1+\frac{1}{%
\alpha }}\text{ }\left( 1-y\right) y^{\frac{1}{\alpha }}}{m^{2}}=2\frac{4^{%
\frac{1}{\alpha }}\text{ }\left( 1-y\right) y^{\frac{1}{\alpha }}}{m^{1-%
\frac{1}{\alpha }}}\leq \frac{4\text{ }\left( 1-y\right) y^{\frac{1}{\alpha }%
}}{m^{1-\frac{1}{\alpha }}}.
\end{equation*}

Subcase $b).$ Let $r\geq s+1$ and assume now that $s<r-\left[ \left(
r+1\right) +\frac{\left( s+1\right) ^{2}}{m}\right] ^{\frac{1}{\alpha }}.$
In subcase $(a)$, the function $\phi \left( r\right) =r-\left[ \left(
r+1\right) +\frac{\left( s+1\right) ^{2}}{m}\right] ^{\frac{1}{\alpha }}$
defined by equation (\ref{fi_fonk}) is non-decreasing with respect to $r\geq
0$. Hence, it could be concluded that there is a maximum value $\bar{r}\in
\left\{ 1,2,...\right\} $, fulfilling the condition $\bar{r}-\left[ \left( 
\bar{r}+1\right) +\frac{\left( s+1\right) ^{2}}{m}\right] ^{\frac{1}{\alpha }%
}<s.$ So, for $r_{1}=\bar{r}+1$ we get $r_{1}-\left[ \left( r_{1}+1\right) +%
\frac{\left( s+1\right) ^{2}}{m}\right] ^{\frac{1}{\alpha }}\geq s.$ Also,
by the way we chose $\bar{r}$ it implies that $r_{1}\leq 2\left( s+1\right)
. $ Thus, similar to subcase $(a)$ we have 
\begin{eqnarray*}
M_{\bar{r}+1,m,s}\left( y\right) &=&m_{\bar{r}+1,m,s}\left( y\right) \left( 
\frac{\bar{r}+1}{m+\bar{r}+1}-y\right) \\
&\leq &\frac{4^{\frac{1}{\alpha }}\text{ }\left( 1-y\right) y^{\frac{1}{%
\alpha }}}{m^{1-\frac{1}{\alpha }}}.
\end{eqnarray*}%
Since $r_{1}>r_{1}-\left[ \left( r_{1}+1\right) +\frac{\left( s+1\right) ^{2}%
}{m}\right] ^{\frac{1}{\alpha }}\geq s,$ following $r_{1}\geq s+1$ and from
Lemma \ref{Lemma 4.3}$\ (i)$ as a result, $M_{\bar{r}+1,m,s}\left( y\right)
\geq M_{\bar{r}+2,m,s}\left( y\right) \geq ...$ . Hence, we have $%
M_{r,m,s}\left( y\right) <\frac{4^{\frac{1}{\alpha }}\text{ }\left(
1-y\right) y^{\frac{1}{\alpha }}}{m^{1-\frac{1}{\alpha }}}$ for each $r\in
\left\{ \bar{r}+1,\bar{r}+2,...\right\} .$

Case $2).$Subcase $a).$ Let $r\leq s$ and assume first that $s\leq r+\left[
r+\frac{s^{2}}{m}\right] ^{\frac{1}{\alpha }}\ $or $s-\left[ r+\frac{s^{2}}{m%
}\right] ^{\frac{1}{\alpha }}\leq r.$ We have 
\begin{eqnarray*}
\text{$M$}_{r,m,s}\left( y\right) &=&m_{r,m,s}\left( y\right) \left( y-\frac{%
r}{m+r}\right) \leq \frac{s+1}{m+s+1}-\frac{r}{m+r} \\
&\leq &\frac{s+1}{m+s+1}-\frac{s-\left[ r+\frac{s^{2}}{m}\right] ^{\frac{1}{%
\alpha }}}{m+s-\left[ r+\frac{s^{2}}{m}\right] ^{\frac{1}{\alpha }}} \\
&=&\frac{%
\begin{array}{c}
\left( m+s\right) s+m+s-s\left[ r+\frac{s^{2}}{m}\right] ^{\frac{1}{\alpha }%
}-\left[ r+\frac{s^{2}}{m}\right] ^{\frac{1}{\alpha }}-\left( m+s\right) s-s
\\ 
+m\left[ r+\frac{s^{2}}{m}\right] ^{\frac{1}{\alpha }}+s\left[ r+\frac{s^{2}%
}{m}\right] ^{\frac{1}{\alpha }}+\left[ r+\frac{s^{2}}{m}\right] ^{\frac{1}{%
\alpha }}%
\end{array}%
}{\left( m+s+1\right) \left[ m+s-\left[ r+\frac{s^{2}}{m}\right] ^{\frac{1}{%
\alpha }}\right] } \\
&=&\frac{m\left[ \left[ r+\frac{s^{2}}{m}\right] ^{\frac{1}{\alpha }}+1%
\right] }{\left( m+s+1\right) \left[ m+s-\left[ r+\frac{s^{2}}{m}\right] ^{%
\frac{1}{\alpha }}\right] }.
\end{eqnarray*}%
Taking into account that $r\leq s$ and $s\leq \frac{my}{1-y}$ we have%
\begin{eqnarray*}
\text{$M$}_{r,m,s}\left( y\right) &\leq &\frac{m\left[ \left[ s+\frac{s^{2}}{%
m}\right] ^{\frac{1}{\alpha }}+1\right] }{\left( m+s+1\right) \left[ m+s-%
\left[ s+\frac{s^{2}}{m}\right] ^{\frac{1}{\alpha }}\right] } \\
&\leq &\frac{m\left[ \left[ \frac{my}{1-y}+\frac{\left( \frac{my}{1-y}%
\right) ^{2}}{m}\right] ^{\frac{1}{\alpha }}+1\right] }{\left( m+s+1\right) %
\left[ m+s-\left[ \frac{my}{1-y}+\frac{\left( \frac{my}{1-y}\right) ^{2}}{m}%
\right] ^{\frac{1}{\alpha }}\right] } \\
&=&\frac{m\left[ \left[ \frac{my}{1-y}\frac{1}{1-y}\right] ^{\frac{1}{\alpha 
}}+1\right] }{\left( m+s+1\right) \left[ m+s-\left[ \frac{my}{1-y}\frac{1}{%
1-y}\right] ^{\frac{1}{\alpha }}\right] } \\
&=&\frac{m\left[ \left( my\right) ^{\frac{1}{\alpha }}+\left( 1-y\right) ^{%
\frac{2}{\alpha }}\right] }{\left( m+s+1\right) \left[ \left( m+s\right)
\left( 1-y\right) ^{\frac{2}{\alpha }}-\left( my\right) ^{\frac{1}{\alpha }}%
\right] }.
\end{eqnarray*}%
Also because of $\frac{m+y-1}{1-y}=\frac{m}{1-y}-1\leq s+m$ and we have%
\begin{eqnarray*}
\text{$M$}_{r,m,s}\left( y\right) &\leq &\frac{m\left[ \left( my\right) ^{%
\frac{1}{\alpha }}+\left( 1-y\right) ^{\frac{2}{\alpha }}\right] }{\left( 
\frac{m}{1-y}-1+1\right) \left[ \left( \frac{m}{1-y}-1\right) \left(
1-y\right) ^{\frac{2}{\alpha }}-\left( my\right) ^{\frac{1}{\alpha }}\right] 
} \\
&=&\frac{\left( 1-y\right) \left[ \left( my\right) ^{\frac{1}{\alpha }%
}+\left( 1-y\right) ^{\frac{2}{\alpha }}\right] }{m\left( 1-y\right) ^{\frac{%
2}{\alpha }-1}-\left( my\right) ^{\frac{1}{\alpha }}-\left( 1-y\right) ^{%
\frac{2}{\alpha }}} \\
&\leq &\frac{\left( 1-y\right) \left[ \left( my\right) ^{\frac{1}{\alpha }}+1%
\right] }{m\left( 1-y\right) ^{\frac{2}{\alpha }-1}-\left( my\right) ^{\frac{%
1}{\alpha }}-\left( 1-y\right) ^{\frac{2}{\alpha }}}.
\end{eqnarray*}%
In the last inequality we used $\left( 1-y\right) ^{\frac{2}{\alpha }}\leq
1. $

Since $s\geq 1$ and $m\geq 4$ also $y\in \left[ \frac{s}{m+s},\frac{s+1}{%
m+s+1}\right] ,$ then $\frac{4}{5}\leq my$ it follows $\left( \frac{4}{5}%
\right) ^{\frac{1}{\alpha }}\leq \left( my\right) ^{\frac{1}{\alpha }}$, so $%
1\leq \left( \frac{5}{4}\right) ^{\frac{1}{\alpha }}$ $\left( my\right) ^{%
\frac{1}{\alpha }}.$ We obtain $\left( my\right) ^{\frac{1}{\alpha }}+1\leq
\left( my\right) ^{\frac{1}{\alpha }}+\left( \frac{5}{4}\right) ^{\frac{1}{%
\alpha }}\left( my\right) ^{\frac{1}{\alpha }}=\left[ 1+\left( \frac{5}{4}%
\right) ^{\frac{1}{\alpha }}\right] \left( my\right) ^{\frac{1}{\alpha }%
}\leq \frac{5}{2}\left( my\right) ^{\frac{1}{\alpha }}.$

Also, let's denote the function $\lambda $ as 
\begin{equation*}
\lambda \left( y\right) :=m\left( 1-y\right) ^{\frac{2}{\alpha }-1}-\left(
my\right) ^{\frac{1}{\alpha }}-\left( 1-y\right) ^{\frac{2}{\alpha }}.
\end{equation*}%
We have 
\begin{eqnarray*}
\lambda ^{\prime }\left( y\right) &=&m\left( \frac{2}{\alpha }-1\right)
\left( -1\right) \left( 1-y\right) ^{\frac{2}{\alpha }-2}-\frac{m}{\alpha }%
\left( my\right) ^{\frac{1}{\alpha }-1}-\frac{2}{\alpha }\left( -1\right)
\left( 1-y\right) ^{\frac{2}{\alpha }-1} \\
&=&\frac{m\left( \alpha -2\right) }{\alpha \left( 1-y\right) ^{2-\frac{2}{%
\alpha }}}-\frac{m}{\alpha \left( my\right) ^{1-\frac{1}{\alpha }}}+\frac{2}{%
\alpha \left( 1-y\right) ^{1-\frac{2}{\alpha }}} \\
&=&\frac{m\left( \alpha -2\right) \left( my\right) ^{1-\frac{1}{\alpha }%
}-m\left( 1-y\right) ^{2-\frac{2}{\alpha }}+2\left( 1-y\right) }{\alpha
\left( 1-y\right) ^{2-\frac{2}{\alpha }}\left( my\right) ^{1-\frac{1}{\alpha 
}}}.
\end{eqnarray*}%
So, because of $m\geq 4,$ the function $\lambda $ is nonincreasing for $%
\alpha =2.$ And for $\alpha \geq 3$ the function $\lambda $ is nondecreasing
when $0<y<1$.

Let $\alpha =2,$ then we have $\lambda \left( y\right) \geq \lambda \left(
1\right) ,$ 
\begin{equation*}
\lim\limits_{y\rightarrow 1}\left[ \frac{m}{\left( 1-y\right) ^{1-\frac{2}{%
\alpha }}}-\left[ \left( my\right) ^{\frac{1}{\alpha }}+\left( 1-y\right) ^{%
\frac{2}{\alpha }}\right] \right] \geq \lim\limits_{y\rightarrow 1}\frac{m}{2%
}.
\end{equation*}%
Now, let $\alpha \geq 3,$ then we have $\lambda \left( y\right) \geq \lambda
\left( 0\right) =m-1$ and since $m\geq 4$, then clearly it follows$\ \lambda
\left( 0\right) \geq \frac{m}{2}.$

Therefore, in both case, we have $\lambda \left( y\right) \geq \frac{m}{2}$
or $\frac{1}{\lambda \left( y\right) }\leq \frac{2}{m}.$

Using what we found above, we have 
\begin{equation*}
\text{$M$}_{r,m,s}\left( y\right) \leq \frac{5}{2}\frac{\left( 1-y\right)
\left( my\right) ^{\frac{1}{\alpha }}}{\frac{m}{2}}=5\frac{\left( 1-y\right)
y^{\frac{1}{\alpha }}}{m^{1-\frac{1}{\alpha }}}.
\end{equation*}%
Subcase $b).$ Now let $r\leq s$ and assume that $s>r+\left[ r+\frac{s^{2}}{m}%
\right] ^{\frac{1}{\alpha }}.$ Let $\tilde{r}\in \left\{ 1,2,...,s\right\} $
be the minimum value such that $\tilde{r}+\left[ \tilde{r}+\frac{s^{2}}{m}%
\right] ^{\frac{1}{\alpha }}>s$ or $\tilde{r}>s-\left[ \tilde{r}+\frac{s^{2}%
}{m}\right] ^{\frac{1}{\alpha }}.$ Then $r_{2}=\tilde{r}-1$ provides $r_{2}+%
\left[ r_{2}+\frac{s^{2}}{m}\right] ^{\frac{1}{\alpha }}\leq s$ and similar
to subcase $a)$ we get 
\begin{eqnarray*}
M_{\tilde{r}-1,m,s}\left( y\right) &=&m_{\tilde{r}-1,m,s}\left( y\right)
\left( y-\frac{\tilde{r}-1}{m+\tilde{r}-1}\right) \\
&\leq &\frac{s+1}{m+s+1}-\frac{\tilde{r}-1}{m+\tilde{r}-1} \\
&\leq &\frac{s+1}{m+s+1}-\frac{s-\left[ \tilde{r}+\frac{s^{2}}{m}\right] ^{%
\frac{1}{\alpha }}-1}{m+s-\left[ \tilde{r}+\frac{s^{2}}{m}\right] ^{\frac{1}{%
\alpha }}-1} \\
&=&\frac{%
\begin{array}{c}
\left( m+s\right) s+m+s-s\left[ \tilde{r}+\frac{s^{2}}{m}\right] ^{\frac{1}{%
\alpha }}-\left[ \tilde{r}+\frac{s^{2}}{m}\right] ^{\frac{1}{\alpha }%
}-s-1-\left( m+s\right) s-s \\ 
+m\left[ \tilde{r}+\frac{s^{2}}{m}\right] ^{\frac{1}{\alpha }}+s\left[ 
\tilde{r}+\frac{s^{2}}{m}\right] ^{\frac{1}{\alpha }}+\left[ \tilde{r}+\frac{%
s^{2}}{m}\right] ^{\frac{1}{\alpha }}+m+s+1%
\end{array}%
}{\left( m+s+1\right) \left[ m+s-\left[ \tilde{r}+\frac{s^{2}}{m}\right] ^{%
\frac{1}{\alpha }}-1\right] } \\
&=&\frac{m\left[ \left[ \tilde{r}+\frac{s^{2}}{m}\right] ^{\frac{1}{\alpha }%
}+2\right] }{\left( m+s+1\right) \left[ m+s-\left[ \tilde{r}+\frac{s^{2}}{m}%
\right] ^{\frac{1}{\alpha }}-1\right] }.
\end{eqnarray*}%
And since $\tilde{r}\leq s$ and $s\leq \frac{my}{1-y}$ then similar to above
subcase we have 
\begin{eqnarray*}
M_{\tilde{r}-1,m,s}\left( y\right) &\leq &\frac{m\left[ \left[ s+\frac{s^{2}%
}{m}\right] ^{\frac{1}{\alpha }}+2\right] }{\left( m+s+1\right) \left[ m+s-%
\left[ s+\frac{s^{2}}{m}\right] ^{\frac{1}{\alpha }}-1\right] } \\
&\leq &\frac{m\left[ \left[ \frac{my}{1-y}\frac{1}{1-y}\right] ^{\frac{1}{%
\alpha }}+2\right] }{\left( m+s+1\right) \left[ m+s-\left[ \frac{my}{1-y}%
\frac{1}{1-y}\right] ^{\frac{1}{\alpha }}-1\right] } \\
&=&\frac{m\left[ \left( my\right) ^{\frac{1}{\alpha }}+2\left( 1-y\right) ^{%
\frac{2}{\alpha }}\right] }{\left( m+s+1\right) \left[ \left( m+s-1\right)
\left( 1-y\right) ^{\frac{2}{\alpha }}-\left( my\right) ^{\frac{1}{\alpha }}%
\right] }.
\end{eqnarray*}%
Also because of $\frac{m}{1-y}-1\leq m+s$ then we get%
\begin{eqnarray*}
M_{\tilde{r}-1,m,s}\left( y\right) &\leq &\frac{m\left[ \left( my\right) ^{%
\frac{1}{\alpha }}+2\left( 1-y\right) ^{\frac{2}{\alpha }}\right] }{\left( 
\frac{m}{1-y}-1+1\right) \left[ \left( \frac{m}{1-y}-2\right) \left(
1-y\right) ^{\frac{2}{\alpha }}-\left( my\right) ^{\frac{1}{\alpha }}\right] 
} \\
&=&\frac{\left( 1-y\right) \left[ \left( my\right) ^{\frac{1}{\alpha }%
}+2\left( 1-y\right) ^{\frac{2}{\alpha }}\right] }{m\left( 1-y\right) ^{%
\frac{2}{\alpha }-1}-\left( my\right) ^{\frac{1}{\alpha }}-2\left(
1-y\right) ^{\frac{2}{\alpha }}} \\
&\leq &\frac{\left( 1-y\right) \left[ \left( my\right) ^{\frac{1}{\alpha }}+2%
\right] }{m\left( 1-y\right) ^{\frac{2}{\alpha }-1}-\left( my\right) ^{\frac{%
1}{\alpha }}-2\left( 1-y\right) ^{\frac{2}{\alpha }}}.
\end{eqnarray*}%
Since $1\leq \left( \frac{5}{4}\right) ^{\frac{1}{\alpha }}\left( my\right)
^{\frac{1}{\alpha }}$ then we obtain $\left( my\right) ^{\frac{1}{\alpha }%
}+2\leq \left( my\right) ^{\frac{1}{\alpha }}+2\left( \frac{5}{4}\right) ^{%
\frac{1}{\alpha }}\left( my\right) ^{\frac{1}{\alpha }}=\left[ 1+2\left( 
\frac{5}{4}\right) ^{\frac{1}{\alpha }}\right] \left( my\right) ^{\frac{1}{%
\alpha }}.$ Also, let's denote the function $\mu $ as 
\begin{equation*}
\mu \left( y\right) :=m\left( 1-y\right) ^{\frac{2}{\alpha }-1}-\left(
my\right) ^{\frac{1}{\alpha }}-2\left( 1-y\right) ^{\frac{2}{\alpha }}.
\end{equation*}%
We have 
\begin{eqnarray*}
\mu ^{\prime }\left( y\right) &=&m\left( \frac{2}{\alpha }-1\right) \left(
-1\right) \left( 1-y\right) ^{\frac{2}{\alpha }-2}-\frac{m}{\alpha }\left(
my\right) ^{\frac{1}{\alpha }-1}-2\frac{2}{\alpha }\left( -1\right) \left(
1-y\right) ^{\frac{2}{\alpha }-1} \\
&=&\frac{m\left( \alpha -2\right) }{\alpha \left( 1-y\right) ^{2-\frac{2}{%
\alpha }}}-\frac{m}{\alpha \left( my\right) ^{1-\frac{1}{\alpha }}}+\frac{4}{%
\alpha \left( 1-y\right) ^{1-\frac{2}{\alpha }}} \\
&=&\frac{m\left( \alpha -2\right) \left( my\right) ^{1-\frac{1}{\alpha }%
}-m\left( 1-y\right) ^{2-\frac{2}{\alpha }}+4\left( 1-y\right) }{\alpha
\left( 1-y\right) ^{2-\frac{2}{\alpha }}\left( my\right) ^{1-\frac{1}{\alpha 
}}}.
\end{eqnarray*}%
So, because of $m\geq 4,$ the function $\mu \left( y\right) $ for $\alpha =2$%
, is nonincreasing and for $\alpha \geq 3,$ is nondecreasing when $0<y<1$.

Let $\alpha =2,$ we have $\mu \left( y\right) \geq \mu \left( 1\right) ,$ 
\begin{equation*}
\lim\limits_{y\rightarrow 1}\left[ \frac{m}{\left( 1-y\right) ^{1-\frac{2}{%
\alpha }}}-\left[ \left( my\right) ^{\frac{1}{\alpha }}+2\left( 1-y\right) ^{%
\frac{2}{\alpha }}\right] \right] \geq \lim\limits_{y\rightarrow 1}\frac{m}{2%
}.
\end{equation*}%
Now, let $\alpha \geq 3,$ we have $\mu \left( y\right) \geq \mu \left(
0\right) =m-2$ and since $m\geq 4$, then$\ \mu \left( 0\right) \geq \frac{m}{%
2}.$

In both case we have $\mu \left( y\right) \geq \frac{m}{2}$ or $\frac{1}{\mu
\left( y\right) }\leq \frac{2}{m}.$

Therefore, we get%
\begin{eqnarray*}
M_{\tilde{r}-1,m,s}\left( y\right) &\leq &\frac{\left[ 1+2\left( \frac{5}{4}%
\right) ^{\frac{1}{\alpha }}\right] \left( 1-y\right) \left( my\right) ^{%
\frac{1}{\alpha }}}{\frac{m}{2}} \\
&\leq &2\left[ 1+2\left( \frac{5}{4}\right) ^{\frac{1}{\alpha }}\right] 
\frac{\left( 1-y\right) y^{\frac{1}{\alpha }}}{m^{1-\frac{1}{\alpha }}} \\
&\leq &8\frac{\left( 1-y\right) y^{\frac{1}{\alpha }}}{m^{1-\frac{1}{\alpha }%
}}.
\end{eqnarray*}%
In the light of Lemma \ref{Lemma 4.3}, $(ii)$, following $M_{\tilde{r}%
-1,m,s}\left( y\right) \geq M_{\tilde{r}-2,m,s}\left( y\right) \geq ...\geq
M_{0,m,s}\left( y\right) .$ Thus, we obtain 
\begin{equation*}
\text{$M$}_{r,m,s}\left( y\right) \leq 8\frac{\left( 1-y\right) y^{\frac{1}{%
\alpha }}}{m^{1-\frac{1}{\alpha }}}.
\end{equation*}
\end{proof}

\begin{remark}
The study \cite{max prod MKZ} showed that the degree of approximation for $%
Z_{m}^{(M)}(g)(y)$ operators is $\frac{\left( 1-y\right) \sqrt{y}}{\sqrt{m}}$%
, using the modulus of continuity. Nevertheless, thanks to Lemma \ref{Lemma
4.3} based on Theorem \ref{Theorem5.1}, we have shown that the degree of
approximation is $\frac{\left( 1-y\right) y^{\frac{1}{\alpha }}}{m^{1-\frac{1%
}{\alpha }}}.$ For sufficiently big $\alpha $, the expression $1/m^{1-\frac{1%
}{\alpha }}$ tends to $1/m$. Consequently, this choice of \ $\alpha $
improves the order of approximation, since $1-\frac{1}{\alpha }\geq \frac{1}{%
2}$ for $\alpha =2,3,...$ .
\end{remark}


\begin{thebibliography}{99}
\bibitem{max-prod-berns} Bede, B., Coroianu, L., Gal, S. G., \textit{%
Approximation and shape preserving properties of the Bernstein operator of
max-product kind,} Intern. J. Math. and Math. Sci., (2009), Article ID
590589, 26 pages, doi:10.1155/2009/590589.

\bibitem{max-prod-Szasz} Bede, B., Coroianu, L., Gal, S. G., \textit{%
Approximation and shape preserving properties of the nonlinear Favard-Sz\'{a}%
sz-Mirakjan operator of max-product kind, }Filomat, \textbf{24(3)} (2010),
55-72.

\bibitem{max-prod Baskakov} Bede, B., Coroianu, L., Gal, S. G., \textit{%
Approximation and shape preserving properties of the nonlinear Baskakov
operator of max-product kind,} Studia Univ. Babe\c{s}-Bolyai (Cluj), Ser.
Math., \textbf{55(4)} (2010), 193-218.

\bibitem{Max-prod- Bleimann-Butz-Hahn} Bede, B., Coroianu, L., Gal, S. G., 
\textit{Approximation and shape preserving properties of the nonlinear
Bleimann-Butzer-Hahn operators of max-product kind,} Carol., Comment. Math.
Univ. Carolin., \textbf{51(3)} (2010), 397-415.

\bibitem{max prod MKZ} Bede, B., Coroianu, L., Gal, S. G., \textit{%
Approximation and shape preserving properties of the nonlinear Meyer--K\"{o}%
nig and Zeller operator of max-product kind}, Numer. Funct. Anal. Opt., 
\textbf{31(3)}, (2010), 232--253.

\bibitem{max-prod kitap} Bede, B., Coroianu, L., Gal, S. G., Approximation
by Max-Product Type Operators, Springer International Publishing,
Switzerland (2016).

\bibitem{max-prod-berns-szazs} Bede, B., Coroianu, L., Gal, S. G., \textit{%
Approximation by Nonlinear Bernstein and Favard-Sz\'{a}sz-Mirakjan operators
of max-product kind, }Journal of Concrete and Applicable Mathematics, 
\textbf{8(2)} (2010), 193-207.

\bibitem{approx by pseudo} Bede, B., Nobuhara, H., Da\v{n}kov\'{a}, M., Di
Nola, A., \textit{Approximation by pseudo-linear operators,} Fuzzy Sets and
Syst., \textbf{159} (2008) 804-820.

\bibitem{max-prod-shepard} Bede, B., Nobuhara, H., Fodor, J., Hirota, K.,, 
\textit{Max-product Shepard approximation operators,} JACIII, \textbf{10},
494-497, (2006).

\bibitem{cheneysharma} E.W. Cheney, A. Sharma, Bernstein power series,
Canad. J. Math. 16 (1964) 241--252.

\bibitem{better berns} \c{C}it, S., Do\u{g}ru, O., \textit{On better
approximation order for the nonlinear Bernstein operator of max-product kind}%
, Filomat, \textbf{38(13)}, 4767-4774, (2024). doi:
https://doi.org/10.2298/FIL2413767C

\bibitem{better BHH} \c{C}it, S., Do\u{g}ru, O., \textit{On better
approximation order for the nonlinear Bleimann-Butzer-Hahn operator of
max-product kind}, KJM., \textbf{9(2),} 225-245, (2023),
doi:10.22034/KJM.2023.357068.2641

\bibitem{better Szasz} \c{C}it, S., Do\u{g}ru, O., \textit{On better
approximation order for the nonlinear Favard-Sz\'{a}sz-Mirakjan operator of
max-product kind}, Mat. Vesn., doi: 10.57016/MV-nkMX9624

\bibitem{Gal} Gal, S. G., \textit{Shape-Preserving Approximation by Real and
Complex Polynomials, Birkh\"{a}user}, Boston-Basel-Berlin, 2008.

\bibitem{Meyer Konig Zeller} W. Meyer-K\"{o}nig, K. Zeller, \textit{%
Bernsteinsche potenzreihen}, Studia Math., \textbf{19} (1960), 89--94.
\end{thebibliography}
\end{document}